\documentclass[12pt]{amsart}
\usepackage[T1,T5]{fontenc} 
\usepackage{amsfonts, amsbsy, amsmath, amssymb}

\usepackage[T1]{fontenc}
\usepackage[utf8]{inputenc}
\usepackage{lmodern}
\usepackage{amsmath, amsthm, amssymb,amscd, mathrsfs, amsfonts, mathtools,tikz-cd}
\usepackage{hyperref}
\usepackage{euler}
\usepackage{times}
\usepackage[all]{xy}
\usepackage{todonotes}
\usepackage{xcolor}
\usepackage[lite]{amsrefs}

\renewcommand{\PrintDOI}[1]{\href{http://dx.doi.org/\detokenize{#1}}{doi: \detokenize{#1}}%
  \IfEmptyBibField{pages}{, (to appear in print)}{}}

\def\commutatif{\ar@{}[rd]|{\circlearrowleft}}

\newtheorem{thm}{Theorem}[section]
\newtheorem{pro}[thm]{Proposition}
\newtheorem{lem}[thm]{Lemma}
\newtheorem{cor}[thm]{Corollary}
\newtheorem{conj}[thm]{Conjecture}

\theoremstyle{definition}
\newtheorem{definition}[thm]{Definition}

\newtheorem*{ack}{Acknowledgments}

\theoremstyle{remark}
\newtheorem{rmk}[thm]{Remark}

\allowdisplaybreaks

\newcommand{\Id}{\operatorname{id}}
\newcommand\Hom{\operatorname{Hom}}

\newcommand{\C}{\mathbb{C}}
\newcommand{\R}{\mathbb{R}}

\newcommand{\Z}{\mathbb{Z}}

\newcommand{\I}{\mathcal{I}}
\newcommand{\A}{\widehat{A}}
\newcommand{\Alpha}{\widehat{\alpha}}
\renewcommand{\phi}{\varphi}
\renewcommand{\epsilon}{\varepsilon}
\newcommand{\Mid}{_{\operatorname{mid}}}
\newcommand{\supp}{\operatorname{supp}}
\newcommand{\PHI}{\widetilde{\Phi}}
\newcommand{\PSI}{\widetilde{\Psi}}
\newcommand{\inv}{^{-1}}

\usepackage{bbm}

\def\Aut{\operatorname{Aut}}
\def\End{\operatorname{End}}

\def\Core{\operatorname{Core}}
\def\Alex{\operatorname{Alex}}

\hypersetup{
  colorlinks = true,
  urlcolor = blue,
  linkcolor = blue,
  citecolor = red,
  pdfauthor = {Mohamed Elhamdadi, Lực Ta, and Bryce Virgin},
  pdfkeywords = {Quandles, ring quandle},
  pdftitle = {Fourier Analysis and Idempotents in Quandle Algebras},
  pdfsubject = { racks, quandles},
  pdfpagemode = UseNone
}

\title{Fourier Analysis and Idempotents in Quandle Algebras}

\author{Mohamed Elhamdadi}
\address{Department of Mathematics and Statistics, University of South Florida, Tampa, FL, USA}
\email{emohamed@usf.edu}

\author{{\fontencoding{T5}\selectfont L\d\uhorn c} Ta}
\address{Department of Mathematics, University of Pittsburgh, Pittsburgh, PA, USA}
\email{ldt37@pitt.edu}

\author{Bryce Virgin}
\address{Department of Mathematics and Statistics, University of South Florida, Tampa, FL, USA}
\email{bgvirgin@usf.edu}

\begin{document}

\maketitle 

\begin{abstract}
    In 2023, the first author, Nunez, Singh, and Swain \cite{MR4565221} formulated an analogue of Kaplansky's idempotent conjecture for integral quandle rings. In this paper, we develop a new Fourier-analytic approach to quandle rings and use it to establish the conjecture for two major classes of quandles: Takasaki quandles, including dihedral quandles, and medial commutative quandles.
    
A central contribution of the paper is the introduction of Fourier analysis on Alexander quandles, which provides a new framework for studying quandle rings and, in particular, their idempotents.
Using this Fourier-analytic framework, we also prove that a previously known sufficient condition for the existence of counterexamples to the conjecture is in fact necessary, thereby resolving a problem of Jabłonowski \cite{counterexample}. 
\end{abstract}

\section{Introduction}

In 1982, Joyce \cite{Joyce} and Matveev \cite{Matveev} independently introduced non-associative algebraic structures called \emph{quandles} to construct invariants of knots and links \cite{EN}. Since then, quandles and more general objects called \emph{racks} have enjoyed applications to mathematical physics \cite{EB}, algebraic geometry \cite{ta}, and the theory of the set-theoretic Yang--Baxter equation \cite{BES}, among other areas of mathematics. 

In the same way that one studies groups $G$ via their group rings $\mathbf k[G]$, one may study quandles $X$ via their \emph{quandle rings} $\mathbf k[X]$, which Bardakov, Passi, and Singh \cite{BPS} introduced in 2019. In 2023, the first author, Nunez, Singh, and Swain \cite{MR4565221} posed the following quandle-theoretic analogue of Kaplansky's idempotent conjecture. The conjecture is related to certain knot invariants and has become one of the most widely studied conjectures in quandle theory in recent years; see \cites{bardakov-elhamdadi-2026, BES, BPS-2, MR4579329, churchill, ta-2} for previous work on the conjecture.

\begin{conj}\label{Main conjecture}
The integral quandle ring $\Z[X]$ of a semi-latin quandle $X$ has only trivial idempotents. In particular, the integral quandle ring of a finite latin quandle has only trivial idempotents.
\end{conj}

Recently, Jab\l onowski \cite{counterexample} proposed a class of counterexamples to Conjecture \ref{Main conjecture}, the smallest of which is the Alexander quandle determined by multiplication by $18$ in $\Z/37\Z$. To better understand these and potential other counterexamples, it is important to determine which classes of quandles \emph{do} satisfy Conjecture \ref{Main conjecture}; this is Question 9.2 in \emph{op.\ cit.} Since Jab\l onowski's proof passes from $\Z$ to $\C$ and uses character theory, it is also important to study complex quandle algebras $\C[X]$ from the point of view of representation theory when $X=\Alex(A,\phi)$ is an Alexander quandle. The results of this paper fall within this program.

In this paper, we introduce Fourier analysis on $\sigma$-compact, locally compact Alexander quandles, which we use to prove Conjecture \ref{Main conjecture} for two major classes of quandles and solve an open problem of Jab\l onowski. In more detail, we prove Conjecture \ref{Main conjecture} for semi-latin Takasaki quandles (Theorem \ref{thm:takasaki})---including all latin dihedral quandles---and all medial commutative quandles (Theorem \ref{thm:comm}). On the other hand, Jab\l onowski posed the question of whether his sufficient condition for certain Alexander quandles to fail Conjecture \ref{Main conjecture} is also a necessary condition; see Question 9.2 in \emph{op.\ cit.} We give a positive answer to this problem (Theorem \ref{thm:jab}). As a consequence, we obtain Conjecture \ref{Main conjecture} for Alexander quandles on $\Z/p\Z$ for primes $p\not\equiv 1\pmod{12}$ (Corollary \ref{cor:primes}).

Theorem \ref{thm:takasaki} vastly improves on several results in the literature, including \cite{MR4565221}*{Prop.\ 3.11}. 
Since every ordered abelian group is torsion-free (hence 2-torsionless), Theorem \ref{thm:takasaki} applies to all Takasaki quandles of ordered abelian groups. Therefore, Theorem \ref{thm:takasaki} solves \cite{bardakov-elhamdadi-2026}*{Ques.\ 3.2} in the case that $\mathbf{k}=\Z$. Moreover, Theorem \ref{thm:takasaki} (or, more specifically, Proposition \ref{pro:finite-takasaki}) proves Conjecture \ref{Main conjecture} for all latin dihedral quandles $R_n=T(\Z/n\Z)$. This vastly generalizes \cite{BPS-2}*{Prop.\ 4.3}, \cite{MR4565221}*{Cor.\ 3.12}, and \cite{churchill}*{Sec.\ 4}. Similarly, Theorem \ref{thm:comm} solves \cite{bardakov-elhamdadi-2026}*{Ques.\ 5.9} and generalizes \cite{churchill}*{Sec.\ 4}.

The structure of this paper is as follows. 
In Section \ref{review}, we review the definitions of quandles and quandle rings. In Section \ref{sec:alexander}, we introduce Fourier analysis on $\sigma$-compact, locally compact Alexander quandles, which we apply to study quandle algebras of finite latin Alexander quandles over $\C$ and $\Z$. We apply these Fourier-analytic results in Sections \ref{sec:takasaki}, \ref{sec:medial-comm}, and \ref{sec:jab} to prove Theorems \ref{thm:takasaki}, \ref{thm:comm}, and \ref{thm:jab}, respectively.

\section{Review of quandles and quandle rings}\label{review}

\subsection{Quandles}
We start by recalling some basic definitions and examples.  More details can be found in \cites{Joyce, Matveev, EN, BES}.

\begin{definition}
A \emph{rack} $(X,*)$ is a set  with a binary operation such that each right multiplication $S_x\colon X \rightarrow X$, $y \mapsto y*x$, is an automorphism of $(X,*)$.  If furthermore, $x*x=x$ for all $x \in X$, then  $(X,*)$ is called a \emph{quandle}.
\end{definition}
\begin{definition}
    Let $(X,*)$ and $(Y,\cdot)$ be racks. A map $f\colon X\to Y$ is called a \emph{rack homomorphism} if $f(w*x)=f(w)\cdot f(x)$ for all $w,x\in X$.
\end{definition}
The following are a few examples of quandles.
\begin{itemize}
\item
 A quandle is called \emph{trivial} if $x*y=x$ for all $ x ,y\in X$.

\item
For a group $G$ and an automorphism $\phi \in {\rm Aut}(G)$, the binary operation $x*y=\phi (xy^{-1}) y $ makes $G$ into a quandle called a \emph{generalized Alexander quandle}. Although the inversion map $\iota(x):=x^{-1}$ is an automorphism if and only if $G$ is abelian, choosing $\phi:=\iota$ yields the \emph{core quandle} $\Core(G)$ even if $G$ is nonabelian.

\item
For a group $G$, the operation $x*y=y^{-1}xy$ gives a quandle structure on $G$ called a \emph{conjugation quandle}.

\end{itemize}

A quandle $(X, *)$ is called {\it latin} (resp.\ \emph{semi-latin}) if every left multiplication map $L_y\colon X \to X$, which is defined by $L_y(x)=y*x$ for $y \in X$, is a bijection (resp.\ injection).  Every latin quandle is clearly semi-latin, but not conversely. For example, $\Core(\mathbb{Z})$ is semi-latin but not latin.  

\subsection{Quandle rings}
Let $(X, *)$ be a quandle and $\mathbf{k}$ an integral domain with unity. To each $x \in X$, we assign a unique symbol $e_x$. 
Let $\mathbf{k}[X]$ be the set of all formal expressions of the form $\sum_{x \in X }  a_x e_x$, where $a_x \in \mathbf{k}$ such that all but finitely many $a_x=0$.  The set $\mathbf{k}[X]$ has a free $\mathbf{k}$-module structure with basis $\{e_x \mid x \in X \}$ and admits a product given by 
 $$ \Big( \sum_{x \in X }  a_x e_x \Big) \Big( \sum_{ y \in X }  b_y e_y \Big)
 =   \sum_{x, y \in X } a_x b_y e_{x * y},$$
where $x, y \in X$ and $a_x, b_y \in \mathbf{k}$. This turns $\mathbf{k}[X]$ into a generally non-associative $\mathbf k$-algebra called the {\it quandle ring} or \emph{quandle algebra} of $X$ with coefficients in $\mathbf{k}$. Note that $\mathbf k[X]$ is associative if and only if $X$ is a trivial quandle. This construction defines a functor from quandles to non-associative $\mathbf k$-algebras. The quandle $X$ can be identified as a subset of $\mathbf{k}[X]$ via the natural map $x \mapsto e_x$. 
\par
The surjective ring homomorphism $\varepsilon\colon \mathbf{k}[X] \rightarrow \mathbf{k}$ given by $$\varepsilon \Big(\sum_{x \in X }  \alpha_x e_x \Big)=\sum_{x \in X }  \alpha_x$$
is called the {\it augmentation map}. The kernel of $\varepsilon $ is a two-sided ideal of $\mathbf{k}[X]$ called the {\it augmentation ideal} of $\mathbf{k}[X]$. 
For more on quandle rings, see \cite{BES}. 

\subsubsection{Idempotents in quandle rings}
Let $X$ be a quandle and $\mathbf{k}$ an integral domain with unity.  A nonzero element $u \in \mathbf{k}[X]$ is called an {\it idempotent} if $u^2=u$. We denote the set of all idempotents of $\mathbf{k}[X]$ by $\mathcal{I}\big(\mathbf{k}[X]\big)$, that is,
$$
\mathcal{I}\big(\mathbf{k}[X] \big)=\big\{ u  \in \mathbf{k}[X]\; | \; u^2=u\big \}.
$$
It is clear that the basis elements $\{e_x \mid x \in X\}$ are idempotents of $\mathbf{k}[X]$, and we refer them as {\it trivial idempotents}.  A \emph{nontrivial idempotent}  is an element of $\I(\mathbf{k}[X])$ that is not of the form $e_x$ for any $x \in X$. 
Note that since $\epsilon\colon\mathbf k[X]\to\mathbf k$ is a ring homomorphism and $\mathbf k$ is an integral domain, we have $\epsilon(u)\in\{0,1\}$ for all nonzero idempotents $u\in\mathcal I(\mathbf k[X])$.

\section{Alexander quandles}\label{sec:alexander}

In this section, we study Alexander quandles of $\sigma$-compact, locally compact abelian groups by way of Fourier analysis. Later in this paper, we use these results to obtain restrictions on the coefficients $\alpha_x$ of each nonzero idempotents $u\in\Z[\Alex(A,\phi)]$ for certain Alexander quandles $\Alex(A,\phi)$.

Given an additive abelian group $A$ and an automorphism $\phi\in\Aut(A)$, recall that the \emph{Alexander quandle} $\Alex(A,\phi)$ is the set $A$ equipped with the medial quandle operation
\[
x*y\coloneq \phi(x)+(\Id-\phi)(y).
\]
Note that $\Alex(A,\phi)$ is latin if and only if the endomorphism $\Id-\phi$ is invertible.
It is worth mentioning that every nonempty medial latin quandle is isomorphic to an Alexander quandle (see \cite{ta-2026}*{Prop.\ 4.1}), and all of the hypotheses in this statement are necessary.

For example, if $\phi$ is the inversion map $x\mapsto -x$, then $\Alex(A,\phi)$ is called a \emph{Takasaki quandle} and denoted by $T(A)$. Note that $T(A)=\Core(A)$, and $R_n\coloneq T(\Z/n)$ is called the \emph{dihedral quandle} of order $n$. 

\subsection{Analysis on Alexander quandles}
In the following, let $A$ be a $\sigma$-compact, locally compact abelian group, and let $\phi\in\Aut(A)$. Let $\mu$ be a Haar measure of $A$, and let $\Delta_A\colon\Aut(A)\to\R_{>0}$ be the function that sends each automorphism $\phi\in\Aut(A)$ to its modulus. Given $y\in A$, recall that $S_y\in\Aut(\Alex(A,\phi))$ and $L_y\colon A\to A$ denote the right multiplication map $x\mapsto x*y$ and the left multiplication map $x\mapsto y*x$, respectively. 

\begin{pro}\label{pro:l1}
    For each $\phi\in\Aut(A)$, the algebra $L^1(A,\mu,\C)$ has a well-defined product
    \begin{align*}
        (\alpha \beta)(x) &:= \Delta_A\phi^{-1}\int_{ A}(\alpha \circ S_y^{-1})(x) \beta(y)d\mu(y)\\
        &= \Delta_A\phi^{-1}\int_{ A}\alpha (\phi^{-1}(x)+(\Id-\phi^{-1})(y) ) \beta(y)d\mu(y).
    \end{align*}
    If moreover $\Alex(A,\phi)$ is latin, then in fact \[ (\alpha \beta)(x) = \Delta_A(\Id-\phi)^{-1} \int_{ A}\alpha (y')   \beta(L_{y'}^{-1}(x))d\mu(y'). \]
\end{pro}

\begin{proof}
    The product is well-defined because
    \begin{align*}
        \int_A |\alpha \beta|(x) d\mu(x) &= \Delta_A\phi^{-1}\int_A  \left|\int_{ A}(\alpha \circ S_y^{-1})(x)\beta(y)d\mu(y)  \right| d\mu(x)\\
        &\leq\Delta_A\phi^{-1}\int_A   \int_{ A}\left|(\alpha \circ S_y^{-1})(x)\beta(y) \right| d\mu(x)   d\mu(y)\\
        &=\Delta_A\phi^{-1}\int_A   \int_{ A}|\alpha (x')||\beta(y)| d\mu(\phi(x') +(\text{Id}-\phi)(y) )   d\mu(y)\\
        &=\Delta_A\phi^{-1}\int_A   \int_{ A}|\alpha (x')||\beta(y)| d\mu(\phi(x')   )   d\mu(y)\\
        &=\int_A   \int_{ A}|\alpha (x')||\beta(y)| d\mu(x'   )   d\mu(y)   = \Vert \alpha \Vert_1 \Vert \beta \Vert_1 < \infty.
    \end{align*}
    Our use of the Fubini--Tonelli theorem in the second line is justified by the $\sigma$-compactness of $A$ (and consequent $\sigma$-finiteness of $(A,\mu)$). 

    If moreover $\Alex(A,\phi)$ is latin, then 
    \begin{align*}
        (\alpha \beta)(x) &= \Delta_A\phi^{-1}\int_{ A}(\alpha \circ S_y^{-1})(x) \beta(y)d\mu(y)\\
        &=\Delta_A\phi^{-1}\int_{ A}(\alpha \circ S_{L_{y'}^{-1}(x)}^{-1})(x) \beta(L_{y'}^{-1}(x))d\mu(L_{y'}^{-1}(x))\\
        &= \Delta_A\phi^{-1}\int_{ A}(\alpha \circ S_{L_{y'}^{-1}(x)}^{-1})(x) \beta(L_{y'}^{-1}(x))d\mu((\Id -\varphi)^{-1}(x - \varphi(y')))\\
        &= \Delta_A(\Id-\phi)^{-1} \int_{ A}\alpha (y')   \beta(L_{y'}^{-1}(x))d\mu(y'),
    \end{align*}
    as desired.
\end{proof}

\subsubsection{A key identity}
Given an additive abelian group $A$, let $\A\coloneq\Hom(A,S^1)$ be its Pontryagin dual. Given $\alpha\in L^1(A,\mu,\C)$, let $\Alpha\colon\A\to\C$ denote the Fourier transform of $\alpha$. Proposition \ref{pro:l1} leads to the following key identity.

\begin{thm} 
    Let $\Alex(A,\phi)$ be a $\sigma$-compact, locally compact latin Alexander quandle, and let $\alpha,\beta\in L^1(A,\mu,\mathbb{C}) $. Then
    \[
    \widehat{\alpha \beta} (\chi)=\Alpha (\chi\circ\phi)\hat{\beta} (\chi\circ(\Id-\phi))
    \]
    for all $\chi\in\A$.
\end{thm}

\begin{proof}
    Fix $\chi\in\A$. Using the definition of the Fourier transform and Proposition \ref{pro:l1}, we compute
    \begin{align*}
        \widehat{\alpha\beta}(\chi)&=\int_{A}(\alpha\beta)(x)\overline{\chi(x)} d\mu(x) \\
        &=\Delta_A(\Id-\phi)^{-1} \int_{A}\int_{ A}\alpha (y)   \beta(L_{y}^{-1}(x))d\mu(y)\overline{\chi(x)} d\mu(x) \\
        &=\Delta_A(\Id-\phi)^{-1} \int_{A}\int_{ A}\alpha (y)   \beta(L_{y}^{-1}(x))d\mu(y)\overline{\chi(x)} d\mu(\phi(y)+(\Id-\phi)(L_y^{-1}(x)))\\
        &=  \int_{A}\int_{ A}\alpha (y)   \beta(L_{y}^{-1}(x))d\mu(y)\overline{\chi(\phi(y))}\overline{\chi((\Id-\phi)(L_y^{-1}(x)))} d\mu(  L_y^{-1}(x))\\
        &=  \int_{A}\int_{ A}\alpha (y)   \beta(z)d\mu(y)\overline{\chi(\phi(y))}\overline{\chi((\Id-\phi)(z))} d\mu(  z)\\
        &=\left(\int_{  A}\alpha (y)\overline {\chi(\phi(y))} d\mu(y)\right) \left(\int_{  A}\beta(z)\overline {\chi((\Id-\phi)(z))} d\mu(z)\right)\\
        &=\Alpha(\chi\circ\phi)\hat{\beta}(\chi\circ(\Id-\phi)).
    \end{align*}
  Our use of the Fubini--Tonelli theorem in the sixth line is justified by Proposition \ref{pro:l1}.
\end{proof}

\begin{cor}\label{cor:l1}
    Let $\Alex(A,\varphi)$ be a $\sigma$-compact, locally compact latin Alexander quandle, and let $\alpha \in L^1(A,\mu,\mathbb{C}) $ be idempotent. Then
\[\Alpha(\chi) = \Alpha(\chi\circ\varphi) \Alpha(\chi\circ (\Id - \varphi))\]
for all $\chi\in\A$.
\end{cor}

\subsection{Finite Alexander quandles}
Now, let $A$ be a finite abelian group equipped with the discrete topology and the counting measure, which is a Haar measure. Then the quandle algebra $\C[\Alex(A,\phi)]$ is canonically isomorphic (as a nonassociative $\C$-algebra) to $L^1(A,\mu,\mathbb{C})$ with the product defined in Proposition \ref{pro:l1}.

We exploit this isomorphism as follows. Given an element $u\in \C[\Alex(A,\phi)]$, write $u=\sum_{x\in A}\alpha_xe_x$ and define $\alpha\colon A\to\C$ by $x\mapsto \alpha_x$. Then the Fourier transform of $\alpha$ is
\begin{equation}\label{eq:ft}
    \Alpha\colon \A\to\C,\qquad \chi\mapsto \sum_{x\in A}\alpha_x\overline{\chi(x)},
\end{equation}
and Corollary \ref{cor:l1} specializes to the following.

\begin{pro}\label{prop:alex}
    Let $\Alex(A,\phi)$ be a finite latin Alexander quandle, and let $u\in \C[\Alex(A,\phi)]$ and $\Alpha\colon \A\to\C$ be as above. If $u$ is idempotent, then
    \[
    \Alpha (\chi)=\Alpha (\chi\circ\phi)\Alpha (\chi\circ(\Id-\phi))
    \]
    for all $\chi\in\A$.
\end{pro}

Here is another useful auxiliary result, which we prove using Plancherel's identity for finite abelian groups.

\begin{lem}\label{lem:parseval}
    Let $\Alex(A,\phi)$ be a finite Alexander quandle, and let $u\in \I(\Z[\Alex(A,\phi)])$ and $\Alpha\colon \A\to\C$ be as above. If $|\Alpha(\chi)|\in\{0,1\}$ for all $\chi\in\A$, then $u$ is a trivial idempotent.
\end{lem}

\begin{proof}
    We have to show that $u=e_y$ for some $y\in A$. To that end, let $n$ be the cardinality of the set $\{\chi\in \A\mid \Alpha (\chi)\neq 0\}\subseteq \A$. Combine Plancherel's identity (see \cite{fourier}*{Thm.\ 2.4.3}) and the hypothesis to obtain
    \begin{equation}\label{eq:parseval}
        \sum_{x\in A}\alpha_x^2=\frac{1}{|A|}\sum_{\chi\in \A}|\Alpha (\chi)|^2=\frac{n}{|A|}.
    \end{equation}
    Recall that $A\cong\A$ because $A$ is finite (see, for example, Corollary 2.3.4 in \emph{op.\ cit.}); in particular, $n\leq |A|$. 
    But the left-hand side of \eqref{eq:parseval} is a positive integer, so we must have $n=|A|$. Therefore, since the $\alpha_x$'s are integers, there exists some $y\in A$ such that $\alpha_y=\pm 1$ and $\alpha_x=0$ for all $x\neq y$.
    If $\alpha_y=-1$, then the idempotence of $u$ yields
    \[
    -e_y=u=u^2=(-e_y)^2=e_y,
    \]
    which is absurd. Hence, $\alpha_y=1$, so $u=e_y$.
\end{proof}

\begin{rmk}
    More generally, the statement of Lemma \ref{lem:parseval} still holds if we replace $\Z$ with the ring of integers $\mathcal O_K$ of a totally real number field $K$; the proof is the same. Indeed, one can use the field trace and the AM-GM inequality to show that if $\sum^n_{i=1}\alpha^2_i=1$ in $\mathcal O_K$, then all but one of the $\alpha_i$'s vanish, and the remaining one equals $\pm 1$.
\end{rmk}

\section{Takasaki quandles}\label{sec:takasaki}

In this section, we prove Conjecture \ref{Main conjecture} for all semi-latin Takasaki quandles.

Let $A$ be an additive abelian group. Recall that the Takasaki quandle of $A$ is defined to be $T(A)=\Alex(A,-\Id)$.  
Evidently, $T(A)$ is semi-latin if and only if $A$ has no 2-torsion. In particular, if $A$ is finitely generated, then $T(A)$ is semi-latin if and only if the torsion subgroup $T\leq A$ has odd order. 

We first consider the case in which $A$ is finite. 

\begin{pro}\label{pro:finite-takasaki}
    Finite latin Takasaki quandles satisfy Conjecture \ref{Main conjecture}.
\end{pro}

\begin{proof}
    Let $A$ be a finite abelian group of odd order, and let $u\in\I(\Z[T(A)])$ be a nonzero idempotent. We have to show that $u=e_y$ for some $y\in A$. Let $\A$ be the Pontryagin dual of $A$. Given $\psi\in\A$ and $k\in\Z$, denote by $\psi^k\in\A$ the homomorphism $x\mapsto \psi(x)^k=\psi(k\cdot x)$.
    Since $T(A)=\Alex(A,-\Id)$, Proposition \ref{prop:alex} states that 
    \begin{equation}\label{eq:conjugate}
        \Alpha (\psi)=\Alpha (\psi^{-1})\Alpha (\psi^2)=\overline{\Alpha (\psi)}\Alpha (\psi^2)
    \end{equation}
    for all $\psi\in\A$,
    where in the second equality we used the fact that $\alpha_x\in\Z\subset\R$ for all $x\in A$. 
    
    By Lemma \ref{lem:parseval}, it suffices to show that $|\Alpha (\chi)|\in\{0,1\}$ for all $\chi\in \A$. To that end, consider the integer $k\coloneq (|A|+1)/2$. The algebraic integers $\Alpha (\chi)$ and $\Alpha (\chi^k)$ are Galois conjugates, so the former is nonzero if and only if the latter is nonzero. On the other hand, applying \eqref{eq:conjugate} with $\psi\coloneq \chi^k$ yields
    \[
    \Alpha (\chi^k)=\overline{\Alpha (\chi^k)}\Alpha (\chi^{2k})=\overline{\Alpha (\chi^k)}\Alpha (\chi).
    \]
    Thus, if $\Alpha (\chi)\neq 0$, then we must have $|\Alpha (\chi)|=1$, as desired.
\end{proof}

To reduce the general case to Proposition \ref{pro:finite-takasaki}, we use the following technical lemma.

\begin{lem}\label{lem:takasaki}
    Let $G$ be a finitely generated abelian group, and let $S\subseteq G$ be a finite subset. Then $G$ has a finite quotient $\pi\colon G\twoheadrightarrow H$ such that the restriction $\pi|_S$ is injective. Moreover, if $G$ has no 2-torsion, then we can choose for the order of $H$ to be odd.
\end{lem}

\begin{proof}
    By the structure theorem for finitely generated abelian groups, we can assume that $G=\mathbb{Z}^r\oplus T$ for some $r\geq 0$ and some finite abelian group $T$. By Cauchy's theorem, $|T|$ is odd if $G$ has no 2-torsion. 
    
    Since $S$ is finite, we can pick an odd integer $n$ such that\[n> \max_{1\leq i\leq r}\{|v_i-w_i|: \mathbf{v},\mathbf{w}\in S\}.\] 
    Let $H\coloneq (\mathbb{Z}/n\mathbb{Z})^r\oplus T$, and consider $H$ as a finite quotient group $\pi\colon G\twoheadrightarrow H$. We claim that $\pi|_S$ is injective. Indeed, suppose that $\pi(\mathbf{v})=\pi(\mathbf{w})$ for some $\mathbf{v},\mathbf{w}\in S$. Then, for all $1\leq i\leq r$, the nonnegative integer $|v_i-w_i|$ is divisible by $n$ and strictly less than $n$, so it must equal $0$. Hence, $\mathbf{v}=\mathbf{w}$.
\end{proof}

Given a quandle $X$ and an element $u\in \Z[X]$, say $u=\sum_{x\in X}\alpha_xe_x$, we define the \emph{support} of $u$ to be the finite subset
\[
\supp(u)\coloneq\{x\in X\mid \alpha_x\neq 0\}\subseteq X.
\]

\begin{thm}\label{thm:takasaki}
    All semi-latin Takasaki quandles satisfy Conjecture \ref{Main conjecture}.
\end{thm}

\begin{proof}
    Let $A$ be an abelian group with no 2-torsion, and let $u\in\I(\Z[T(A)])$ be a nonzero idempotent. We have to show that $u$ is trivial. The subgroup $G\coloneq\langle\supp (u)\rangle\leq A$ is a finitely generated abelian group with no 2-torsion, so Lemma \ref{lem:takasaki} states that $G$ has a finite quotient $\pi\colon G\twoheadrightarrow H$ of odd order such that $\pi|_{\supp (u)}$ is injective. Note that $u\in \I(\Z[T(G)])$.

    View $\pi$ as a quandle epimorphism $\pi\colon T(G)\twoheadrightarrow T(H)$, and consider the induced ring epimorphism $\widetilde{\pi}\colon\Z[T(G)]\twoheadrightarrow\Z[T(H)]$. Since $\pi|_{\supp (u)}$ is injective, the restriction of $\widetilde{\pi}$ to the $\Z$-submodule \[U\coloneq \operatorname{span}_\Z(\supp(u))\leq \Z[T(G)]\] is also injective. Since $u$ is a nonzero idempotent contained in $U$, its image $\widetilde{\pi}(u)$ is a nonzero idempotent in $\Z[T(H)]$. But $T(H)$ is finite and latin, so Proposition \ref{pro:finite-takasaki} states that $\widetilde{\pi}(u)$ is trivial. Since $\widetilde{\pi}|_U$ is injective and induced by the canonical projection $\pi$, this implies that $u$ is also trivial.
\end{proof}

\begin{cor}\label{cor:aut-tak}
    If $A$ is an abelian group with no 2-torsion, then $\Aut(\Z[T(A)])\cong A\rtimes \Aut(A)$.
\end{cor}

\begin{proof}
    Combine Theorem \ref{thm:takasaki} with \cite{bardakov-elhamdadi-2026}*{Prop.\ 2.4} and \cite{bardakov-2017}*{Thm.\ 4.2}.
\end{proof}

See \cite{ta-2} for applications of Corollary \ref{cor:aut-tak}.

\section{Medial commutative quandles}\label{sec:medial-comm}

Recall that commutative quandles are latin. Idempotents of quandle algebras of certain commutative quandles over $\Z$ and various other rings were previously studied in \citelist{\cite{bardakov-elhamdadi-2026}\cite{ta-2}}. In this section, we prove Conjecture \ref{Main conjecture} for all medial commutative quandles. 

In the following, let $\mathbb{D}\coloneq\Z[1/2]$ denote the ring of dyadic rational numbers. In Corollary 4.2 of \cite{ta-2026}, the second author proved that every nonempty medial commutative quandle is isomorphic to a \emph{midpoint quandle} $A\Mid$, that is, an Alexander quandle $A\Mid\coloneq \Alex(A,\phi)$ such that $A$ is a $\mathbb{D}$-module and $\phi$ is multiplication by $1/2$. Conversely, every midpoint quandle is medial and commutative. 

The structure of the proof is similar to that of Theorem \ref{thm:takasaki}; once again, we begin by using Fourier analysis to prove the finite case.

\begin{pro}\label{pro:finite-comm}
    Finite medial commutative quandles satisfy Conjecture \ref{Main conjecture}. 
\end{pro}

\begin{proof}
    Certainly, Conjecture \ref{Main conjecture} holds for the empty quandle. That aside, let $A$ be a finite abelian group of odd order, and let $u\in\I(\Z[A\Mid])$ be a nonzero idempotent. We have to show that $u$ is trivial. As before, denote by $\phi\in\Aut(A)$ the automorphism $x\mapsto x/2$, so that $A\Mid=\Alex(A,\phi)$. Let $\A$ be the Pontryagin dual of $A$, and define
    \[
    \Phi\colon \A\to\A,\qquad \chi\mapsto\chi\circ\phi.
    \]
    Then $\Phi$ is a permutation (even an automorphism) of $\A$, and Proposition \ref{prop:alex} states that
    \begin{equation}\label{eq:recursion}
        \Alpha(\Phi(\chi))^2=\Alpha(\chi)
    \end{equation}
    for all $\chi\in\A$.

    By Lemma \ref{lem:parseval}, it suffices to show that $|\Alpha (\chi)|\in\{0,1\}$ for all $\chi\in \A$. Since $\A$ is finite, the orbit of each $\chi\in\A$ under the action of $\Phi$ is finite, say $\{\chi,\dots,\Phi^{m-1}(\chi)\}$ with $\Phi^m(\chi)=\chi$. In particular,
    \[
    \Alpha(\chi)^{2^m}=\Alpha(\Phi^{m}(\chi))^{2^m}=\Alpha(\chi),
    \]
    where the second equality follows after applying \eqref{eq:recursion} $m$ times. 
    This implies our claim.
\end{proof}

To reduce the general case to Proposition \ref{pro:finite-comm}, we use the following technical lemma. The lemma may be viewed as an analogue of Lemma \ref{lem:takasaki} for $\mathbb{D}$-modules instead of $\Z$-modules.

\begin{lem}\label{lem:comm}
    Let $M$ be a finitely generated $\mathbb{D}$-module, and let $S\subseteq M$ be a finite subset. Then there exists a $\mathbb{D}$-module epimorphism $\pi\colon M\twoheadrightarrow N$ such that $\pi|_S$ is injective and $|N|<\infty$.
\end{lem}

\begin{proof}
    Note that $\mathbb{D}$ is a PID, and quotients of $\mathbb{D}$ by nonzero ideals are finite. Therefore, by the structure theorem for finitely generated modules over PIDs, we can assume that $M=\mathbb{D}^r\oplus T$ for some $r\geq 0$ and some $\mathbb{D}$-module $T$ with $|T|<\infty$.
    Write $S=\{(\mathbf{d}_i, t_i)\mid 1\leq i\leq |S|\}$ with $\mathbf{d}_i\in\mathbb{D}^r$ and $t_i\in T$ for all $1\leq i\leq |S|$. Define the sets \[X\coloneq\{(i,j)\in\mathbb{N}^2\mid 1 \leq i<j\leq |S|\text{ and }t_i=t_j\}\subset\mathbb{N}^2\]
    and
    \[Y\coloneq\{\mathbf{y}_{ij}\coloneq \mathbf{d}_i-\mathbf{d}_j\mid (i,j)\in X\}\subset\mathbb{D}^r.\] Note that $X$ and $Y$ are finite because $S$ is finite. Moreover, $\mathbf{0}\notin Y$. 
    
    It suffices to find an odd integer $n\in\mathbb{N}$ such that $\mathbf{y}_{ij}\not\equiv \mathbf{0}\pmod{n\mathbb{D}^r}$ for all $\mathbf{y}_{ij}\in Y$, since then we can take $\pi$ to be the canonical projection onto the quotient module $N\coloneq (\mathbb{D}/n\mathbb{D})^r\oplus T$. 
    Since $X$ is finite, we can pick $k\in\mathbb{N}$ large enough that $Y\subset 2^{-k}\mathbb{Z}^r$. Consider the integer vectors \[\mathbf{z}_{ij}\coloneq 2^k \mathbf{y}_{ij}\in \mathbb{Z}^r\setminus\{\mathbf{0}\}.\] For each such vector $\mathbf{z}_{ij}$, let $n_{ij}\in\mathbb{N}$ be the greatest common divisor of the entries of $\mathbf{z}_{ij}$. Since $Y$ is finite, there exists an odd integer $n$ such that $n\nmid n_{ij}$ for all $(i,j)\in X$. That is, $\mathbf{z}_{ij}\not\equiv \mathbf{0}\pmod{n\mathbb{Z}^r}$ for all $(i,j)\in X$. Since $n$ is odd, this is equivalent to the statement that $\mathbf{y}_{ij}\not\equiv \mathbf{0}\pmod{n\mathbb{D}^r}$ for all $\mathbf{y}_{ij}\in Y$, which is what we wanted to show.
\end{proof}

\begin{rmk}
    Lemma \ref{lem:comm} and its proof easily generalize to $\Z[1/n]$-modules with $n\geq 2$. (The $n=1$ case is already given by Lemma \ref{lem:takasaki}.)
\end{rmk}

The following is the main theorem of this section. 

\begin{thm}\label{thm:comm}
    All medial commutative quandles satisfy Conjecture \ref{Main conjecture}. 
\end{thm}

\begin{proof}
    The proof is nearly identical to the proof of Theorem \ref{thm:takasaki}. The only differences are that we use Proposition \ref{pro:finite-comm} and Lemma \ref{lem:comm} instead of Proposition \ref{pro:finite-takasaki} and Lemma \ref{lem:takasaki}, and we work with $\mathbb{D}$-modules and midpoint quandles instead of abelian groups and Takasaki quandles. The same proof works because every $\mathbb{D}$-module homomorphism is also a homomorphism of the underlying midpoint quandles; cf.\ \cite{ta-2026}*{Prop.\ 5.1}.
\end{proof} 

\begin{cor}
    If $A$ is a $\mathbb{D}$-module, then $\Aut(\Z[A\Mid])\cong\Aut(A\Mid)$.
\end{cor}

\begin{proof}
    Combine Theorem \ref{thm:comm} with \cite{bardakov-elhamdadi-2026}*{Prop.\ 2.4}.
\end{proof}

The upcoming paper \cite{sangare} contains a computation of $\Aut(A\Mid)$ in the case that $A$ is finite. 

\section{Solution to a problem of Jab\l onowski}\label{sec:jab}
In \cite{counterexample}*{Ques.\ 9.2}, Jab\l onowski asked whether his sufficient condition for certain finite latin Alexander quandles to violate Conjecture \ref{Main conjecture} (see Theorem 6.1 in \emph{op.\ cit.}) is also necessary. In this section, we give a positive answer to this problem using our Fourier-analytic results for Alexander quandles.

\subsection{Alexander quandles, revisited}
First, we establish some notation. 
Henceforth, we fix a finite latin Alexander quandle $\Alex(A,\phi)$. Let $\A$ be the Pontryagin dual of $A$, and let $\iota\in\Aut(\A)$ denote the inversion map $\chi\mapsto\overline\chi$. A short calculation shows that
\begin{equation}\label{eq:conjugate2}
    \overline{\hat{f}(\chi)}=(\hat{\overline f}\circ\iota)(\chi)
\end{equation}
for all functions $f\colon A\to\C$, where $\hat f\colon\A\to\C$ denotes the Fourier transform of $f$. 
On the other hand, as in the proof of Proposition \ref{pro:finite-comm}, we have group automorphisms $\Phi,\Psi\in\Aut(\A)$ defined by
\[
\Phi(\chi)\coloneq\chi\circ\phi,\qquad\Psi(\chi)\coloneq\chi\circ(\Id-\phi).
\]
Note that $\Phi$ and $\Psi$ commute, and they also commute with $\iota$ because $\Phi,\Psi\in\Aut(\A)$.

Identify the group algebra $\C[\A]$ as the $\C$-algebra of functions $\A\to\C$. 
Given a group endomorphism $f\in\End(\A)$, let $\Tilde{f}\in\End_\C(\C[\A])$ denote precomposition with $f$ in $\C[\A]$. Let $V\coloneq\ker(\Tilde{\iota}-\Id_{\C[\A]})$ be the $\tilde\iota$-invariant subspace of $\C[\A]$. 
Since $\Phi$ and $\Psi$ commute with $\iota$, both $\PHI$ and $\PSI$ commute with $\Tilde{\iota}$, so $V$ is stable under the map $\PHI+\PSI\in\End_\C(\C[\A])$. Let $M\in\End_\C(V)$ denote the restriction of $\PHI+\PSI$ to $V$. 

Recall that for each element $u\in\C[\Alex(A,\phi)]$, we write $u=\sum_{x\in A}\alpha_xe_x$ and define $\alpha\colon A\to\C$ and $\Alpha\colon\A\to\C$ as in \eqref{eq:ft}. In the following proof, the construction of the function $v\colon\A\to\R$ is inspired by the proof of \cite{counterexample}*{Prop.\ 3.3}.

\begin{pro}\label{prop:specm}
    Let $\Alex(A,\phi)$ be a finite latin Alexander quandle. If $\Z[\Alex(A,\phi)]$ contains a nontrivial idempotent $u=\sum_{x\in A}\alpha_xe_x$, then $1$ is an eigenvalue of $M$.
\end{pro}

\begin{proof}
    Consider the set $S\coloneq\{\chi\in\A\mid\Alpha(\chi)\neq 0\}$. Recall from Proposition \ref{prop:alex} that
    \begin{equation}\label{eq:specm}
        \Alpha(\chi)=\Alpha(\Phi(\chi))\Alpha(\Psi(\chi))
    \end{equation}
    for all $\chi\in\A$. It follows that $\Alpha(\Phi(\chi))\neq 0$ and $\Alpha(\Psi(\chi))\neq 0$ for all $\chi\in S$, so $S$ is stable under $\Phi$ and $\Psi$. Since $\Phi,\Psi\in\Aut(\A)$ and $\A$ is finite, in fact $\Phi(S)=\Psi(S)=S$.
    On the other hand, since $\alpha$ is real-valued (by way of being integer-valued), taking $f\coloneq\alpha$ in \eqref{eq:conjugate2} shows that $\iota(S)=S$ as well. 

    Identify $\C[\A]$ with the $\C$-algebra of functions $\A\to\C$, and define
    \[
    v\colon \A\to\R\subset\C,\qquad v(\chi)\coloneq\begin{cases}
        \log|\Alpha(\chi)|&\text{if }\chi\in S,\\
        0&\text{if }\chi\notin S.
    \end{cases}
    \]
    Note that $v\in V$ because $|\Alpha|$ is $\iota$-invariant. Moreover, $v\neq 0$. Indeed, if $v\equiv 0$, then $|\Alpha(\chi)|\in\{0,1\}$ for all $\chi\in \A$, so Lemma \ref{lem:parseval} states that $u$ is trivial; this contradicts our hypothesis. 
    
    Therefore, it suffices to show that $Mv=v$. 
    For all $\chi\in S$, we have
    \[
    (Mv)(\chi)=v(\Phi(\chi))+v(\Psi(\chi))=\log|\Alpha(\Phi(\chi))\Alpha(\Psi(\chi))|=v(\chi),
    \]
    where in the last equality we have used \eqref{eq:specm}. For all $\chi\in\A\setminus S$, we have $\Phi(\chi),\Psi(\chi)\in\A\setminus  S$ as discussed in the first paragraph, so
    \[
    (Mv)(\chi)=v(\Phi(\chi))+v(\Psi(\chi))=0=v(\chi),
    \]
    as desired. 
\end{proof}

\subsection{Specialization to Jab\l onowski's problem}

Fix an odd prime $p$, and pick $A\coloneq\Z/p\Z$. Given a unit $x\in A^\times$, let $f_x\in\Aut(A)$ denote multiplication by $x$. Fix a unit $a\in A^\times$ such that $a\neq 1$, and consider the latin Alexander quandle $\Alex(A,f_a)$. Given a positive integer $n\geq 1$, let $\zeta_n$ be a primitive $n$th root of unity.

\begin{lem}\label{lem:characters}
    In the above setting, define $V\leq\C[\A]$ and $M\in\End_\C(V)$ as before. The following are equivalent:
    \begin{enumerate}
        \item There exists a nontrivial multiplicative character $\chi$ of $A^\times=(\Z/p\Z)^\times$ such that $\chi(-1)=1$ and $\{\chi(a),\chi(1-a)\}=\{\zeta_6,\zeta_6\inv\}$.
        \item There exists a nontrivial multiplicative character $\chi$ of $A^\times=(\Z/p\Z)^\times$ such that $\chi(-1)=1$ and $\chi(a)+\chi(1-a)=1$.
        \item $1$ is an eigenvalue of $M$.
    \end{enumerate}
    If any of these conditions hold, then $p\equiv 1\pmod{12}$.
\end{lem}

\begin{proof}
    The equivalence of the first two conditions follows from the fact that the only two roots of unity that sum to $1$ are $\zeta_6$ and $\zeta_6\inv$; see \cite{counterexample}*{Lem.\ 4.2}. The first condition also implies the final assertion.

    We prove the equivalence of the second two conditions. Since $A=\Z/p\Z$, there is a canonical isomorphism \[A\xrightarrow{\sim}\A,\quad k\mapsto\omega_k,\quad \omega_k(n)\coloneq\zeta_p^{nk}.\] (See, for example, \cite{fourier}*{Cor.\ 2.3.4}.) 
    Under this identification, $\Phi$, $\Psi$, and $\iota$ act on $\A$ via multiplication by $a$, $1-a$, and $-1$, respectively. 
    
    Identify $\C[\A]$ as the $\C$-algebra of functions $\A\to\C$, and let $G$ be the subgroup of $\Aut(\C[\A])$ generated by $\PHI$, $\PSI$, and $\tilde\iota$. 
    Let $g$ be a generator of $A^\times$, and consider the group of multiplicative characters 
    \[\widehat{A^\times}=\{\chi_k\mid k\in A^\times\},\qquad \chi_k(g)\coloneq\zeta_{p-1}^{k-1}\] 
    of $A^\times$. Extend each $\chi_k\in\widehat{A^\times}$ to a function $\chi_k\colon A\to\C$ by defining $\chi_k(0)\coloneq 0$. Then 
    \[
    \C[\A]=\C\omega_0\oplus\operatorname{span}_\C\{\chi_k\mid k\in A^\times\}
    \]
    is a decomposition of $\C[\A]$ into $G$-invariant subspaces.
    Note that the induced action of $M$ on $\C\omega_0$ is multiplication by $2$. 

    For all $x\in A^\times$, the automorphism $f_x\in\Aut(A)$ induces an automorphism $F_x\in\Aut(\C[\A])$. For all $\chi_k\in\widehat{A^\times}$, we have
    \[
    (F_x\chi_k)(n)=\chi_k(xn)=\chi_k(x)\chi_k(n)
    \]
    for all $n\in A$, so $F_x\chi_k=\chi_k(x)\chi_k$. In particular, taking $x\coloneq a,1-a,-1$ shows that
    \[
    M\chi_k=(\chi_k(a)+\chi_k(1-a))\chi_k,\qquad \tilde\iota\chi_k=\chi_k(-1)\chi_k
    \]
    for all $\chi_k\in\widehat{A^\times}$. Hence, $\{\omega_0,\chi_1\dots,\chi_{p-1}\}$ is an eigenbasis of $\C[\A]$ for the pairwise commuting linear maps $\PHI$, $\PSI$, and $\tilde\iota$. In particular, the $\tilde\iota$-invariant subspace of $\C[\A]$ is
    \[
    V=\ker(\tilde\iota-\Id_{\C[\A]})=\C\omega_0\oplus\bigoplus_{\substack{k\in A^\times\\\chi_k(-1)=1}}\C\chi_k,
    \]
    which consists of eigenvectors of $M$. The corresponding eigenvalues are 
    \[
    \operatorname{Spec}M=\{2\}\cup\{\chi_k(a)+\chi_k(1-a)\mid \chi_k\in \widehat{A^\times},\ \chi_k(-1)=1\}.
    \]
    In particular, the eigenvalues corresponding to $\omega_0$ and the trivial character $\chi_1$ of $A^\times$ are both $2$, so the equivalence of the second two conditions in the claim follows immediately.
\end{proof}

We prove the converse to \cite{counterexample}*{Thm.\ 6.1}, thus solving a problem of Jab\l onowski (see Question 9.2 in \emph{op.\ cit.}). Given $a\in(\Z/p\Z)^\times$, recall that $f_a\in\Aut(\Z/p\Z)$ denotes multiplication by $a$.
\begin{thm}\label{thm:jab}
    Let $p$ be an odd prime, and let $a\in(\Z/p\Z)^\times$. If $a\neq 1$, then the following are equivalent:
    \begin{enumerate}
        \item The quandle ring $\Z[\Alex(\Z/p\Z,f_a)]$ contains a nontrivial idempotent.
        \item Any of the equivalent conditions of Lemma \ref{lem:characters} hold.
        \item The operator $H_a\otimes_\Z\C$ constructed in Section 4 of \emph{op.\ cit.\ }has a nontrivial kernel.
    \end{enumerate}
    If any of these conditions hold, then $p\equiv 1\pmod{12}$.
\end{thm}

\begin{proof}
    It is a special case of Proposition \ref{prop:specm} that (1) implies the third condition of Lemma \ref{lem:characters}. 
    Theorem 4.1 of \emph{op.\ cit.\ }states that the second condition of Lemma \ref{lem:characters} is equivalent to (3). Finally, Theorem 6.1 of \emph{op.\ cit.\ }states that (3) implies (1). The final assertion is taken directly from Lemma \ref{lem:characters}.
\end{proof}

Also, since $\Aut(\Z/p\Z)=\{f_a\mid a\in(\Z/p\Z)^\times\}$ and $\Alex(\Z/p\Z,f_a)$ is latin if and only if $a\neq 1$, we deduce the following.

\begin{cor}\label{cor:primes}
    Let $p$ be a prime such that $p\not\equiv 1\pmod{12}$. Then every latin quandle of the form $\Alex(\Z/p\Z,\phi)$ satisfies Conjecture \ref{Main conjecture}.
\end{cor}

%%%%%%%%%

\begin{ack}
    The initial ideas for this paper grew out of the second author's visit to the University of South Florida in April 2026, where he met the first and third authors. The second author thanks USF for its hospitality and for providing travel funding. We also thank {Micha\l} Jab\l onowski for communicating \cite{counterexample} to us.
\end{ack}

\end{document}